\documentclass[letterpaper,11pt,twoside,keywordsasfootnote,addressatend,noinfoline]{article}
\usepackage{fullpage}
\usepackage[english]{babel}
\usepackage{amssymb}
\usepackage{amsmath}
\usepackage{comment}
\usepackage{theorem}
\usepackage{epsfig}
\usepackage{subfigure}
\usepackage{mathrsfs}
\usepackage{color}
\usepackage{imsart}

\newtheorem{theorem}{Theorem}
\newtheorem{lemma}[theorem]{Lemma}

\newenvironment{proof}{\noindent{\bf Proof.}}{\hspace*{2mm}~$\square$}

\newcommand{\ind}{\mathbf{1}}
\newcommand{\ep}{\epsilon}
\newcommand{\n}{\hspace*{-5pt}}

\DeclareMathOperator{\card}{card}
\DeclareMathOperator{\uniform}{Uniform}

\begin{document}

\begin{frontmatter}
\title     {Limiting behavior of a kindness model}
\runtitle  {Limiting behavior of a kindness model}
\author    {Nicolas Lanchier\thanks{Nicolas Lanchier was partially supported by NSF grant CNS-2000792.} and Max Mercer}
\runauthor {Nicolas Lanchier and Max Mercer}
\address   {School of Mathematical and Statistical Sciences \\ Arizona State University, Tempe, AZ 85287, USA. \\ nicolas.lanchier@asu.edu \\ mamerce1@asu.edu}

\maketitle

\begin{abstract} \ \
 This paper is concerned with a stochastic model for the spread of kindness across a social network.
 Individuals are located on the vertices of a general finite connected graph, and are characterized by their kindness belief.
 Each individual, say~$x$, interacts with each of its neighbors, say~$y$, at rate one.
 The interactions can be kind or unkind, with kind interactions being more likely when the kindness belief of the sender~$x$ is high.
 In addition, kind interactions increase the kindness belief of the recipient~$y$, whereas unkind interactions decrease its kindness belief.
 The system also depends on two parameters modeling the impact of kind and unkind interactions, respectively.
 We prove that, when kind interactions have a larger impact than unkind interactions, the system converges to the purely kind configuration with probability tending to one exponentially fast in the large population limit.
\end{abstract}



\end{frontmatter}


\section{Introduction}
 This paper explores a stochastic model for the spread of kindness/unkindness based on the following two natural assumptions:
 individuals who have experienced kindness/unkindness frequently are more likely to express kindness/unkindness to others, and individuals who are accustomed to kind/unkind interactions are more surprised/affected by unkind/kind interactions.
 The model also depends on two parameters measuring the general impact of kind versus unkind interactions.
 The main objective of this paper is to study the limiting behavior of a spatially explicit stochastic model that describes the spread of kindness/unkindness across a social network and whose dynamics is governed by the two assumptions above. \\
\indent
 Our model is based on the framework of interacting particle systems.
 For an exposition of the early models in this topic, we refer to~\cite{lanchier_2024, liggett_1985, liggett_1999}.
 More precisely, we assume that the individuals are located on the set of vertices of a general finite connected graph~$G = (V, E)$ representing a social network.
 The individuals at vertices~$x$ and~$y$ can interact if and only if they are neighbors~(connected by an edge), which we write~$x \sim y$ from now on.
 The kindness belief of the individuals resulting from the interactions they have experienced from their neighbors is represented by a number in the interval~$[-1, +1]$, with state~$-1$ referring to completely unkind, and state~$+1$ referring to completely kind.
 In particular, the state of the system at time~$t \in [0, \infty)$ is a configuration
 $$ \xi_t : V \longrightarrow [-1, +1] \quad \hbox{where} \quad \xi_t (x) = \hbox{kindness belief at vertex~$x$ at time~$t$}. $$
 To fix the ideas, we assume that the initial kindness beliefs across the network are independent and uniformly distributed in the interval~$[-1, +1]$.
 Each individual, say~$x$, interacts with each of its neighbors, say~$y$, at rate one.
 The interaction is either kind or unkind, more likely to be kind/unkind if the kindness belief of~$x$ is high/low, and we assume that the interaction is
 $$ \hbox{kind with probability} \ 1/2 + \xi (x) / 2 \quad \hbox{and} \quad \hbox{unkind with probability} \ 1/2 - \xi (x) / 2. $$
 As a result of a kind/unkind interaction, the kindness belief of the recipient~$y$ increases/decreases, with the effect of the interaction being more pronounced if the type of the interaction does not align with the current kindness belief of~$y$.
 More precisely, having two parameters~$\mu_+, \mu_- \in [0, 1]$ modeling the sensitivity of the individuals to a kind/unkind interaction,
 $$ \xi_t (y) \to \xi_t (y) + \mu_+ (1 - \xi_t (y)) \quad \hbox{or} \quad \xi_t (y) \to \xi_t (y) - \mu_- (1 + \xi_t (y)), $$
 depending on whether the interaction is kind or unkind, respectively.
 In short, each oriented edge~$\vec{xy}$ of the graph/social network becomes active at rate one, at which time
 $$ \xi_t (y) = \left\{\begin{array}{rcl}
    \xi_{t-} (y) + \mu_+ (1 - \xi_{t-} (y)) & \hbox{with probability} & 1/2 + \xi_{t-} (x) / 2, \vspace*{4pt} \\
    \xi_{t-} (y) - \mu_- (1 + \xi_{t-} (y)) & \hbox{with probability} & 1/2 - \xi_{t-} (x) / 2. \end{array} \right. $$
 Note that, starting with all the individuals in state~$-1$ or~$+1$, the special case~$\mu_- = \mu_+ = 1$ reduces to the popular voter model~\cite{clifford_sudbury_1973, holley_liggett_1975}.
 To fix the ideas~(and stay optimistic), we assume from now on that~$\mu_- < \mu_+$, but in view of the obvious symmetry of our kindness model, our analysis easily extends to the pessimist case~$\mu_- > \mu_+$ by simply replacing~$\xi$ with~$- \xi$.
 Our main result shows that, with probability tending to one in the large population limit, the system converges to the purely kind configuration~$\ind_V$.
 To state our result more precisely, let~$X_t = \sum_{x \in V} \xi_t (x)$ be the process that keeps track of the overall kindness of the system, and
 $$ T_{\ep}^- = \inf \{t : X_t < - \ep N \} \quad \hbox{and} \quad T_{\ep}^+ = \inf \{t : X_t > (1 - \ep) N \}, $$
 where~$N = \card (V)$ refers to the population size.
\begin{theorem}
\label{th:kindness}
 Assume that~$\mu_- < \mu_+$.
\begin{itemize}
\item With probability one, $\xi_t \to + \ind_V$ or \ $\xi_t \to - \ind_V$. \vspace*{4pt}
\item For all~$0 < \ep < 1/2$, there exists~$c_{\ep} > 1$ such that~$P (T_{\ep}^- < T_{\ep}^+) \leq c_{\ep}^{- \ep N / 2}$.
\end{itemize}
\end{theorem}
 In words, the system converges almost surely to the completely kind or to the completely unkind configuration.
 In addition, the probability that the average kindness goes below~$- \ep$ before it goes above~$1 - \ep$ decays exponentially with the population size. 


\section{Proof of the theorem}
 Depending on the configuration of the system, an interaction along~$\vec{xy}$ may increase or decrease the kindness of~$y$ on average.
 The first key to proving the theorem is to observe that, in the optimistic scenario~$\mu_- < \mu_+$, the combined effects of an interaction along~$\vec{xy}$ and an interaction along~$\vec{yx}$ always increase the overall kindness on average.
 In particular, the process~$X_t$ is a submartingale~(Lemma~\ref{lem:X}), so an application of the martingale convergence theorem implies that the process converges almost surely to either the completely kind configuration or the completely unkind configuration~(Lemma~\ref{lem:convergence}).
 More generally, as long as the overall kindness is not too high or too low, we can find~$c > 1$ such that~$c^{- X_t}$ is a supermartingale~(Lemmas~\ref{lem:drift}--\ref{lem:Y}).
 Applying the optional stopping theorem to this supermartingale, we deduce that, with probability tending to one in the large population limit~$N \to \infty$, the average kindness belief across the social network becomes arbitrarily high before decreasing even slightly~(Lemmas~\ref{lem:calculus}--\ref{lem:ost}).
\begin{lemma}
\label{lem:X}
 The process~$X_t = \sum_{x \in V} \xi_t (x)$ is a submartingale.
\end{lemma}
\begin{proof}
 For each oriented edge~$\vec{xy}$, define
 $$ \rho_{\vec{xy}}^t = (1/2 + \xi_t (x) / 2) \,\mu_+ (1 - \xi_t (y)) - (1/2 - \xi_t (x) / 2) \,\mu_- (1 + \xi_t (y)). $$
 Some simple algebra implies that
 $$ \begin{array}{rcl}
    \rho_{\vec{xy}}^{t} + \rho_{\vec{yx}}^{t} & \n = \n &
         (1/2 + \xi_t (x) / 2) \,\mu_+ (1 - \xi_t (y)) - (1/2 - \xi_t (x) / 2) \,\mu_- (1 + \xi_t (y)) \vspace*{4pt} \\ && \hspace*{20pt} + \
         (1/2 + \xi_t (y) / 2) \,\mu_+ (1 - \xi_t (x)) - (1/2 - \xi_t (y) / 2) \,\mu_- (1 + \xi_t (x)) \vspace*{4pt} \\ & \n = \n &
         (1/2)(\mu_+ - \mu_-)((1 + \xi_t (x))(1 - \xi_t (y)) + (1 - \xi_t (x))(1 + \xi_t (y))) \vspace*{4pt} \\ & \n = \n &
         (\mu_+ - \mu_-)(1 - \xi_t (x) \,\xi_t (y)) \geq 0. \end{array} $$
 Now, assuming that~$\vec{xy}$ becomes active at time~$t$,
 $$ \begin{array}{l}
      E (X_t - X_{t-} \,| \,\xi_{t-}) =
      E (\xi_t (y) - \xi_{t-} (y) \,| \,\xi_{t-}) \vspace*{4pt} \\ \hspace*{40pt} =
        (1/2 + \xi_{t-} (x) / 2) \,\mu_+ (1 - \xi_{t-} (y)) - (1/2 - \xi_{t-} (x) / 2) \,\mu_- (1 + \xi_{t-} (y)) = \rho_{\vec{xy}}^{t-}, \end{array} $$
 and since each oriented edge becomes active at rate one,
 $$ \begin{array}{rcl}
    \lim_{s \to 0} \,s^{-1} E (X_{t + s} - X_t \,| \,\xi_t) & \n = \n &
    \sum_{\vec{xy}} \,\rho_{\vec{xy}}^t = \sum_{xy} \,(\rho_{\vec{xy}}^t + \rho_{\vec{yx}}^t) \vspace*{4pt} \\ & \n = \n &
    (\mu_+ - \mu_-) \,\sum_{xy} \,(1 - \xi_t (x) \,\xi_t (y)) \geq 0. \end{array} $$
 This shows that~$X_t$ is a submartingale.
\end{proof} \\ \\
 Using Lemma~\ref{lem:X}, we can deduce that the process converges almost surely to either the completely kind configuration or the completely unkind configuration.
\begin{lemma}
\label{lem:convergence}
 With probability one, $\xi_t \to + \ind_V$ or~$\xi_t \to - \ind_V$.
\end{lemma}
\begin{proof}
 Because~$X_t$ is a submartingale with absolute value bounded by~$N$, it follows from the martingale convergence theorem~(see, e.g., \cite[Theorem 5.2]{lanchier_2017}) that~$X_t$ converges almost surely to a random variable~$X_{\infty}$.
 In particular, for almost all realizations~$\omega$, the sequence~$X_t (\omega)$ converges to a limit corresponding to an absorbing state.
 When~$X_t = N$, or equivalently~$\xi_t = + \ind_V$, all the interactions are kind and do not change the configuration so~$+ \ind_V$ is an absorbing state.
 Similarly,~$- \ind_V$ is an absorbing state.
 In contrast, if~$|X_t| \neq N$, we have the following alternative:
\begin{itemize}
\item
 We have~$\xi_{t-} (x) \in (-1, +1)$ for some~$x \in V$, in which case the probability of a kind/unkind interaction along~$\vec{xy}$ are both positive, and result respectively in
 $$ \xi_t (y) = \left\{\begin{array}{rcl}
    \xi_{t-} (y) + \mu_+ (1 - \xi_{t-} (y)) > \xi_{t-} (y) & \hbox{if} & \xi_{t-} (y) \neq +1, \vspace*{4pt} \\
    \xi_{t-} (y) - \mu_- (1 + \xi_{t-} (y)) < \xi_{t-} (y) & \hbox{if} & \xi_{t-} (y) \neq -1. \end{array} \right. $$
\item
 We have~$\xi_{t-} (z) \in \{-1, +1 \}$ for all~$z \in V$, in which case~$\xi_{t-} (x) = - 1$ and~$\xi_{t-} (y) = + 1$ for some~$x, y \in V$.
 Moving along a path connecting~$x$ and~$y$, which exists because the graph is connected, we can find two neighbors~$x_t \sim y_t$ that satisfy the same conditions.
 An unkind interaction~$x_t \to y_t$ occurs at rate one, and changes the state at~$y_t$.
\end{itemize}
 This shows that~$\pm \ind_V$ are the only two absorbing states, and completes the proof.
\end{proof} \\ \\
 For all~$0 < \ep < 1/2$, let~$T_{\ep} = \min (T_{\ep}^-, T_{\ep}^+)$ be the first time the average kindness across the social network goes below~$- \ep$ or above~$1 - \ep$.
 The next lemma will be used later to prove that the process~$c^{- X_t}$ stopped at time~$T_{\ep}$ is a supermartingale.
\begin{lemma}
\label{lem:drift}
 For all times~$t < T_{\ep}$,
 $$ \xi_t (x_t) \,\xi_t (y_t) \leq 1 - \ep \quad \hbox{for some neighbors} \quad x_t \sim y_y. $$
\end{lemma}
\begin{proof}
 The proof is similar to the proof of the previous lemma.
 Before time~$T_{\ep}$, the states of the individuals cannot be all below~$- \ep$ or all above~$1 - \ep$, which leads to the following alternative:
\begin{itemize}
\item
 We have~$\xi_t (x_t) \in [- \ep, 1 - \ep]$ for some~$x_t \in V$, in which case
 $$ \xi_t (x_t) \,\xi_t (y_t) \leq |\xi_t (x_t)| \leq \max(\ep, 1 - \ep) = 1 - \ep \quad \hbox{for all} \quad y_t \sim x_t. $$
\item
 We have~$\xi_t (z) \notin [- \ep, 1 - \ep]$ for all~$z \in V$, in which case~$\xi_t (x) \leq - \ep$ and~$\xi_t (y) \geq 1 - \ep$ for some~$x, y \in V$.
 Moving along a path connecting~$x$ and~$y$, we can find two neighbors~$x_t \sim y_t$ that satisfy the same conditions, therefore~$\xi_t (x_t) \,\xi_t (y_t) \leq - \ep (1 - \ep) \leq 0 \leq 1 - \ep$.
\end{itemize}
 This completes the proof.
\end{proof} \\ \\
 Let~$Y_t = X_t^{T_{\ep}}$ be the process~$X_t$ stopped at time~$T_{\ep}$.
\begin{lemma}
\label{lem:Y}
 For all~$\ep > 0$, there exists~$c_{\ep} > 1$ such that~$c_{\ep}^{- Y_t}$ is a supermartingale.
\end{lemma}
\begin{proof}
 To begin with, for all~$c > 0$, we define
 $$ \phi_{\vec{xy}}^t (c) = (1/2 + \xi_t (x) / 2) \,c^{- \mu_+ (1 - \xi_t (y))} + (1/2 - \xi_t (x) / 2) \,c^{+ \mu_- (1 + \xi_t (y))} - 1, $$
 and~$\Phi_t (c) = \sum_{\vec{xy}} \,\phi_{\vec{xy}}^t (c)$.
 Observe that~$\Phi_t (1) = \phi_{\vec{xy}}^t (1) = 0$, and that
 $$ \begin{array}{rcl}
     (\phi_{\vec{xy}}^t + \phi_{\vec{yx}}^t)'(1) & \n = \n &
      - (1/2 + \xi_t (x) / 2) \,\mu_+ (1 - \xi_t (y)) + (1/2 - \xi_t (x) / 2) \,\mu_- (1 + \xi_t (y)) \vspace*{4pt} \\ &&
      - (1/2 + \xi_t (y) / 2) \,\mu_+ (1 - \xi_t (x)) + (1/2 - \xi_t (y) / 2) \,\mu_- (1 + \xi_t (x)) \vspace*{4pt} \\ & \n = \n &
      - (\rho_{\vec{xy}}^t + \rho_{\vec{yx}}^t) = - (\mu_+ - \mu_-)(1 - \xi_t (x) \,\xi_t (y)). \end{array} $$
 In particular, by Lemma~\ref{lem:drift}, for all~$t < T_{\ep}$, there exist~$x_t \sim y_t$ such that
 $$ \begin{array}{rcl}
    \Phi_t' (1) & \n = \n &
    \sum_{xy} \,(\phi_{\vec{xy}}^t + \phi_{\vec{yx}}^t)'(1) =
      - (\mu_+ - \mu_-) \,\sum_{xy} \,(1 - \xi_t (x) \,\xi_t (y)) \vspace*{4pt} \\ & \n \leq \n &
      - (\mu_+ - \mu_-)(1 - \xi_t (x_t) \,\xi_t (y_t)) \leq
      - (\mu_+ - \mu_-) \,\ep < 0. \end{array} $$
 Because the function~$c \mapsto \Phi_t (c)$ is smooth, we deduce the existence of~$c_{\ep} > 1$, fixed from now on, such that~$\Phi_t (c_{\ep}) \leq 0$.
 Finally, assuming that~$\vec{xy}$ becomes active at time~$t$,
 $$ \begin{array}{rcl}
      E (c^{- Y_t} - c^{- Y_{t-}} \,| \,\xi_{t-}) & \n = \n &
      E (c^{- Y_{t-} - (\xi_t (y) - \xi_{t-} (y))} - c^{- Y_{t-}} \,| \,\xi_{t-}) \vspace*{4pt} \\ & \n = \n &
        (1/2 + \xi_{t-} (x) / 2) \,c^{- Y_{t-} - \mu_+ (1 - \xi_{t-} (y))} \vspace*{4pt} \\ & \n + \n &
        (1/2 - \xi_{t-} (x) / 2) \,c^{- Y_{t-} + \mu_- (1 + \xi_{t-} (y))} - c^{- Y_{t-}} =
      c^{- Y_{t-}} \phi_{\vec{xy}}^{t-} (c), \end{array} $$
 and since each oriented edge becomes active at rate one,
 $$ \begin{array}{l}
    \lim_{s \to 0} \,s^{-1} E (c_{\ep}^{- Y_{t + s}} - c_{\ep}^{- Y_t} \,| \,\xi_t) =
      c_{\ep}^{- Y_t} \,\sum_{\vec{xy}} \,\phi_{\vec{xy}}^t (c_{\ep}) \vspace*{4pt} \\ \hspace*{80pt} =
      c_{\ep}^{- Y_t} \,\sum_{xy} \,(\phi_{\vec{xy}}^t + \phi_{\vec{yx}}^t)(c_{\ep}) = c_{\ep}^{- Y_t} \,\Phi_t (c_{\ep}) \leq 0. \end{array} $$
 This shows that~$c_{\ep}^{- Y_t}$ is a supermartingale.
\end{proof} \\ \\
 To deduce from Lemma~\ref{lem:Y} the exponential decay of the probability~$P (T_{\ep}^- < T_{\ep}^+)$ in the second part of the theorem, we will need the following technical lemma.
\begin{lemma}
\label{lem:calculus}
 Let~$0 < \ep < 1/2$.
 Then, for all~$c > 1$ close to one,
 $$ \frac{1}{2} < \frac{c}{2 \ln (c)} - \frac{1}{2 c \ln (c)} = \frac{c^2 - 1}{2c \ln (c)} < c^{\ep / 2}. $$
\end{lemma}
\begin{proof}
 Since~$\ln (c) = (c - 1) + o (c - 1)$,
 $$ \lim_{c \to 1} \ \bigg(\frac{c}{2 \ln (c)} - \frac{1}{2 c \ln (c)} \bigg) =
    \lim_{c \to 1} \ \frac{c^2 - 1}{2c (c - 1)} =
    \lim_{c \to 1} \ \frac{c + 1}{2c} = 1, $$
 which proves the first inequality for all~$c$ close to one.
 Note also that
 $$ \begin{array}{l}
    \displaystyle \frac{d}{dc} \bigg(\frac{c^2 - 1}{2c \ln (c)} \bigg) =
    \displaystyle \frac{4c^2 \ln (c) - 2 (c^2 - 1)(\ln (c) + 1)}{(2c \ln (c))^2} =
    \displaystyle \frac{(c^2 + 1) \ln (c) - (c^2 - 1)}{2c^2 \ln^2 (c)} \end{array} $$
 therefore, using that~$\ln (c) = (c - 1) - (c - 1)^2 / 2 + o (c - 1)^2$,
 $$ \begin{array}{rcl}
    \displaystyle \lim_{c \to 1} \ \frac{d}{dc} \bigg(\frac{c^2 - 1}{2c \ln (c)} \bigg) & \n = \n &
    \displaystyle \lim_{c \to 1} \ \frac{(c^2 + 1) \ln (c) - (c^2 - 1)}{2c^2 \ln^2 (c)} \vspace*{4pt} \\ & \n = \n &
    \displaystyle \lim_{c \to 1} \ \frac{(c^2 + 1)((c - 1) - (c - 1)^2/2) - (c^2 - 1)}{2 (c - 1)^2} \vspace*{4pt} \\ & \n = \n &
    \displaystyle \lim_{c \to 1} \ - \frac{c^3 - 3c^2 + 3c - 1}{4 (c - 1)} =
    \displaystyle \lim_{c \to 1} \ - \frac{(c - 1)^2}{4} = 0. \end{array} $$
 Since in addition~$c \mapsto c^{\ep / 2}$ is equal to one and has derivative~$\ep / 2 > 0$ at~$c = 1$, we deduce the second inequality in the lemma for all~$c > 1$ close to one.
 This completes the proof.
\end{proof} \\ \\
 Let~$p_{\ep}^- = P (T_{\ep} = T_{\ep}^-)$ and~$p_{\ep}^+ = P (T_{\ep} = T_{\ep}^+)$.
\begin{lemma}
\label{lem:ost}
 For all~$0 < \ep < 1/2$, there exists~$c_{\ep} > 1$ such that~$p_{\ep}^- \leq c_{\ep}^{- \ep N / 2}$.
\end{lemma}
\begin{proof}
 Because the~$\xi_0 (x)$ are~$\uniform (-1, 1)$ and independent,
 $$ \begin{array}{rcl}
    \displaystyle E (c_{\ep}^{- Y_0}) & \n = \n &
    \displaystyle E (c_{\ep}^{- X_0}) = E \bigg(\prod_{x \in V} c_{\ep}^{- \xi_0 (x)} \bigg) =
    \displaystyle \prod_{x \in V} E (c_{\ep}^{- \xi_0 (x)}) = \bigg(\int_{-1}^1 \frac{c_{\ep}^{- u}}{2} \ du \bigg)^N \vspace*{8pt} \\ & \n = \n &
    \displaystyle \bigg(\bigg[- \frac{c_{\ep}^{- u}}{2 \ln (c_{\ep})} \bigg]_{-1}^1 \bigg)^N =
    \displaystyle \bigg(\frac{c_{\ep}}{2 \ln (c_{\ep})} - \frac{1}{2 c_{\ep} \ln (c_{\ep})} \bigg)^N = \bigg(\frac{c_{\ep}^2 - 1}{2c_{\ep} \ln (c_{\ep})} \bigg)^N. \end{array} $$
 In addition, because the increments of~$X_t$ are bounded by two,
 $$ c_{\ep}^{\ep N} \leq E (c_{\ep}^{- Y_{T_{\ep}}} \,| \,T_{\ep} = T_{\ep}^-) \leq c_{\ep}^{\ep N + 2} \quad \hbox{and} \quad
    c_{\ep}^{- (1 - \ep) N - 2} \leq E (c_{\ep}^{- Y_{T_{\ep}}} \,| \,T_{\ep} = T_{\ep}^+) \leq c_{\ep}^{- (1 - \ep) N}, $$
 which implies that, at time~$T_{\ep}$,
 $$ \begin{array}{rcl}
      E (c_{\ep}^{- Y_{T_{\ep}}}) & \n = \n &
      p_{\ep}^-  \,E (c_{\ep}^{- Y_{T_{\ep}}} \,| \,T_{\ep} = T_{\ep}^-) + p_{\ep}^+ \,E (c_{\ep}^{- Y_{T_{\ep}}} \,| \,T_{\ep} = T_{\ep}^+) \vspace*{4pt} \\ & \n = \n &
      p_{\ep}^-  \,E (c_{\ep}^{- Y_{T_{\ep}}} \,| \,T_{\ep} = T_{\ep}^-) + (1 - p_{\ep}^-) \,E (c_{\ep}^{- Y_{T_{\ep}}} \,| \,T_{\ep} = T_{\ep}^+) \vspace*{4pt} \\ & \n = \n &
      p_{\ep}^- \,(E (c_{\ep}^{- Y_{T_{\ep}}} \,| \,T_{\ep} = T_{\ep}^-) - E (c_{\ep}^{- Y_{T_{\ep}}} \,| \,T_{\ep} = T_{\ep}^+)) + E (c_{\ep}^{- Y_{T_{\ep}}} \,| \,T_{\ep} = T_{\ep}^+) \vspace*{4pt} \\ & \n \geq \n &
      p_{\ep}^- (c_{\ep}^{\ep N} - c_{\ep}^{- (1 - \ep) N}) + c_{\ep}^{- (1 - \ep) N - 2}. \end{array} $$
 Choosing~$c_{\ep} > 1$ as in Lemma~\ref{lem:Y} to ensure that~$c^{- Y_t}$ is a supermartingale, and applying the optional stopping theorem to this process~(see, e.g., \cite[Theorem 5.1]{lanchier_2017}), we deduce that
 $$ \begin{array}{rcl}
      p_{\ep}^- (c_{\ep}^{\ep N} - c_{\ep}^{- (1 - \ep) N}) + c_{\ep}^{- (1 - \ep) N - 2} \leq
      E (c_{\ep}^{- Y_{T_{\ep}}}) \leq
      E (c_{\ep}^{- Y_0}) = ((c_{\ep}^2 - 1) / 2 c_{\ep} \ln (c_{\ep}))^N. \end{array} $$
 Choosing~$c_{\ep} > 1$ smaller if needed to use Lemma~\ref{lem:calculus}, we conclude that
 $$ \begin{array}{rcl}
    \displaystyle p_{\ep}^- \leq \frac{((c_{\ep}^2 - 1) / 2 c_{\ep} \ln (c_{\ep}))^N - c_{\ep}^{- (1 - \ep) N - 2}}{c_{\ep}^{\ep N} - c_{\ep}^{- (1 - \ep) N}} \leq
    \displaystyle \bigg(\frac{c_{\ep}^2 - 1}{2c_{\ep} \ln (c_{\ep})} \bigg)^N c_{\ep}^{- \ep N} && \vspace*{8pt} \\ \hspace*{100pt} =
    \displaystyle \bigg(\frac{c_{\ep}^2 - 1}{2c_{\ep} \ln (c_{\ep}) \,c_{\ep}^{\ep / 2}} \bigg)^N c_{\ep}^{- \ep N / 2} \leq
    \displaystyle c_{\ep}^{- \ep N / 2} & \n \to \n & 0, \end{array} $$
 as the population size~$N \to \infty$.
\end{proof}


\end{document}